\documentclass[12pt]{amsart}

\author{Paul \textsc{Poncet}}
\address{CMAP, \'{E}cole Polytechnique, Route de Saclay, 91128 Palaiseau Cedex, France \\
and INRIA, Saclay--\^{I}le-de-France}
\email{poncet@cmap.polytechnique.fr}
\usepackage{stmaryrd}
\usepackage{amssymb}
\usepackage{amsmath}
\usepackage{graphicx}
\usepackage{t1enc}
\usepackage[dvips]{color}
\usepackage[latin1]{inputenc}
\usepackage{yhmath}
\usepackage{setspace} 
\usepackage{fancybox} 
\usepackage{mathabx}
\usepackage{calrsfs} 
\usepackage{times} 
\usepackage{ifpdf}
\usepackage{stmaryrd}
\usepackage{a4}

\def\twoheaddownarrow{\rlap{$\downarrow$}\raise-.5ex\hbox{$\downarrow$}}
\def\twoheaduparrow{\rlap{$\uparrow$}\raise.5ex\hbox{$\uparrow$}}

\newcommand{\call}{\mathrsfs}

\newcommand{\reels}{\mathbb R}

\newcommand{\ssup}{\gg}

\newtheorem{theorem}{Theorem}[section]
\newtheorem{corollary}[theorem]{Corollary}
\newtheorem{proposition}[theorem]{Proposition}

\theoremstyle{definition}
\newtheorem{definition}[theorem]{Definition}

\newtheorem{remark}[theorem]{Remark}

\newenvironment{acknowledgements}[1][]{\par\vspace{0.5cm}\noindent\textbf{Acknowledgements#1.} }{\par}

\begin{document}

\title{A Decomposition Theorem for Maxitive Measures}
\date{Decembre 28, 2009}

\subjclass{Primary 28B15, 28C15; Secondary 06B35, 03E72, 49J52}

\keywords{max-plus algebra, maxitive measures, idempotent measures, capacities, continuous posets, domains, continuous lattices}

\begin{abstract}
A maxitive measure is the analogue of a finitely additive measure or charge, in which the usual addition is replaced by the supremum operation. Contrarily to charges, maxitive measures often have a density. We show that maxitive measures can be decomposed as the supremum of a maxitive measure with density, and a residual maxitive measure that is null on compact sets under specific conditions.  
\end{abstract}

\maketitle

\vspace{-.7cm}

\section{Introduction}

In the area of idempotent analysis, maxitive measures are usually known as \textit{idempotent measures} after Maslov \cite{Maslov87}. 
Maxitive measures are defined analogously to finitely additive measures with the supremum operation $\oplus$ in place of the addition $+$. 
In the literature, they first appeared in an article by Shilkret \cite{Shilkret71}, and then have been rediscovered and explored for the purpose of capacity theory and large deviations (e.g.\ Norberg \cite{Norberg86}, O'Brien and Vervaat \cite{OBrien91}, Gerritse \cite{Gerritse96}, Puhalskii \cite{Puhalskii01}), idempotent analysis and max-plus (tropical) algebra (e.g.\ Maslov \cite{Maslov87}, Bellalouna \cite{Bellalouna92}, Akian et al.\ \cite{Akian94b}, Del Moral and Doisy \cite{delMoral98}, Akian \cite{Akian99}), fuzzy set theory (e.g.\ Zadeh \cite{Zadeh78}, Sugeno and Murofushi \cite{Sugeno87}, Pap \cite{Pap95}, De Cooman \cite{DeCooman97}, Nguyen et al.\ \cite{Nguyen03}, Poncet \cite{Poncet07}),  optimisation (e.g.\ Barron et al.\ \cite{Barron00}, Acerbi et al.\  \cite{Acerbi02}), or fractal geometry (Falconer \cite{Falconer90}). 



Let $E$ be a nonempty set. A \textit{prepaving} on $E$ is a collection of subsets of $E$ containing the empty set and closed under finite unions. 
Assume in all the sequel that $\call{E}$ is a prepaving on $E$ and that $L$ is a partially ordered set or \textit{poset} wih a bottom element, that we denote by $0$. An $L$-valued \textit{maxitive measure} (resp.\ \textit{completely maxitive measure}) on $\call{E}$ is a map $\nu : \call{E} \rightarrow L$ such that $\nu(\emptyset) = 0$ and, for every finite (resp.\ arbitrary) family $\{G_j\}_{j\in J}$ of elements of $\call{E}$ such that $\bigcup_{j \in J} G_j \in \call{E}$, the supremum of $\{ \nu(G_j) : j\in J \}$ exists and satisfies 
$$
\nu(\bigcup_{j\in J} G_j) = \bigoplus_{j \in J} \nu(G_j). 
$$

If we take for $\call{E}$ the prepaving of all finite subsets of $E$, then every maxitive measure $\nu$ on $\call{E}$ can be written as 
\begin{equation}\label{eqn:den}
\nu(G) = \bigoplus_{x \in G} c^*(x), 
\end{equation}
where $c^*(x) = \nu(\{ x \})$, since $G = \bigcup_{x\in G} \{ x \}$, where the union runs over a finite set. We say that $c^*$ is a \textit{cardinal density} (or a \textit{density} for short) of $\nu$ when Equation~(\ref{eqn:den}) is satisfied. With this simple example, where $E$ need not to be finite for $\nu$ to have a density, we see why compelling $E \in \call{E}$ would be inappropriate.  

In the general case, singletons $\{ x \}$ do not necessarily belong to $\call{E}$, but, as we shall see, one to extend maxitive measures to the whole power set $2^E$ under mild conditions, so it is tempting to consider $c^*(x) := \nu^*(\{x\})$ instead, where $\nu^*$ is the extension of $\nu$ defined as in Equation~(\ref{extension}) below. 
This idea, which appeared in \cite{Heckmann98b, Heckmann98} and \cite{Akian99}, will indeed lead us to necessary and sufficient conditions for a maxitive measure to have a density (see Theorem~\ref{general}). 	

In this article, we are interested in decomposing a maxitive measure into a \textit{regular} part, which is a maxitive measure with a cardinal density, and a \textit{residual} part, also maxitive, and null on compact sets under specific conditions. 
Our motivation comes from possible applications to Radon-Nikod\'ym like theorems for the \textit{Shilkret integral}, developed in \cite{Shilkret71} for maxitive measures, also known as Maslov's \textit{idempotent integral}. 

The results we shall give on maxitive measures are stated in the general case where these measures take their values in a \textit{domain}, the definition of which follows. (For more background on domain theory, see the monograph by Gierz et al.\ \cite{Gierz03}.) 
A subset $F$ of a poset $(P,\leqslant)$ is \textit{filtered} if, for all $x, y \in F$, one can find $z \in F$ such that $z \leqslant x$ and $z \leqslant y$. A \textit{filter} of $P$ is a nonempty filtered subset $F$ of $P$ such that $F = \{ y \in P : \exists x \in F, x \leqslant y \}$. 
We say that $y \in P$ is \textit{way-above} $x\in P$, written $y \ssup x$, if, for every filter $F$ with an infimum $\bigwedge F$, $x \geqslant \bigwedge F$ implies $y \in F$. 
The \textit{way-above relation}, useful for studying lattice-valued upper semicontinuous functions (see Gerritse \cite{Gerritse97} and Jonasson \cite{Jonasson98}), is dual to the usual \textit{way-below relation}, but is more appropriate in our context. Coherently, our notions of continuous posets and domains 
are dual to the traditional ones. 
We thus say that the poset $P$ is \textit{continuous} if $\twoheaduparrow x := \{ y \in P : y \ssup x \}$ is a filter and $x = \bigwedge \twoheaduparrow x$, for all $x \in P$. A \textit{domain} is a continuous poset in which every filter has an infimum. 
A poset $P$ has the \textit{interpolation property} if, for all $x, y \in P$, if $y \ssup x$, there exists some $z \in P$ such that $y \ssup z  \ssup x$. 
In continuous posets it is well known that the interpolation property holds, see e.g.\ \cite[Theorem~I-1.9]{Gierz03}. This is a crucial feature that is behind many important results of the theory. 


Well known examples of domains are $\reels_+$, $\overline{\reels}_+$, and $[0, 1]$. For these posets, the way-above relation coincides with the strict order $>$ (except perhaps at the top). These posets are commonly used as target sets for maxitive measures, and many trials were made for replacing them by more general ordered structures (see Greco \cite{Greco87}, Liu and Zhang \cite{Liu94}, De Cooman et al.\ \cite{deCooman01}, Kramosil \cite{Kramosil05a}). Nevertheless, the importance of the continuity assumption of these structures for applications to idempotent analysis or fuzzy set theory has been identified lately. Pioneers in this direction were Akian (see \cite{Akian95}, \cite{Akian99}) and Heckmann and Huth \cite{Heckmann98b, Heckmann98}. See Lawson \cite{Lawson04b} for a survey on the use of domain theory in idempotent mathematics. See also Poncet \cite{Poncet09} and references therein. 

The paper is organized as follows. 
Sections~\ref{sec:ideals} and \ref{secmax} improve results of \cite{Akian99} and \cite{Heckmann98b, Heckmann98}: we give a representation theorem for maxitive measures, derive the extension theorem cited above, and revisit the problem of finding necessary and sufficient conditions for a maxitive measure to have a cardinal density. 
Section~\ref{decompose} is new and states the announced decomposition theorem. 


\section{Representing Maxitive Measures by Ideals}\label{sec:ideals}

An \textit{ideal} of the prepaving $\call{E}$ is a nonempty subset $\call{I}$ of $\call{E}$ which is stable under finite unions and such that, if $A \subset B$ and $B \in \call{I}$, then $A \in \call{I}$. 

The next proposition, inspired by Nguyen et al.\ \cite{Nguyen97}, provides a generic way of constructing a maxitive measure from a nondecreasing family of ideals. 

\begin{proposition}\label{nguyen bouchon-meunier1}
Let $(\call{I}_t)_{t \in L}$ be some family of ideals of $\call{E}$ such that, for all $G \in \call{E}$, $\{t \in L : G \in \call{I}_t \}$ is a filter with infimum. Define $\nu : \call{E} \rightarrow L$ by  
\begin{equation}\label{ideals0}
\nu(G) = \bigwedge \left\{t \in L : G \in \call{I}_t \right\}. 
\end{equation}
If $(\call{I}_t)_{t \in L}$ is right-continuous, in the sense that $\call{I}_t = \bigcap_{s \ssup t} \call{I}_s$ for all $t \in L$, then $\nu$ is maxitive. 
\end{proposition}

\begin{remark}
Assuming that $\{t \in L : G \in \call{I}_t \}$ is a filter for all $G \in \call{E}$ makes the family $(\call{I}_t)_{t \in L}$ necessarily nondecreasing. 
\end{remark}

\begin{proof}
Let $\nu$ be given by Equation~(\ref{ideals0}). Obviously, $\nu$ is order-preserving, so it remains to show that, for all finite family $\{G_j\}_{j \in J}$ of elements of $\call{E}$, and for every upper bound $m \in L$ of $\{\nu(G_j)\}_{j \in J}$, we get $m \geqslant \nu(\bigcup_{j \in J} G_j)$. 
Let $s \ssup m$. One has $G_j \in \call{I}_s$ for all $j \in J$, thus $\bigcup_{j \in J} G_j \in \call{I}_{s}$. This implies $\bigcup_{j \in J} G_j \in \bigcap_{s \ssup m} \call{I}_s = \call{I}_m$. Eventually $m \geqslant \bigwedge \{r \in L : \bigcup_{j \in J} G_j \in \call{I}_{r} \} = \nu(\bigcup_{j \in J} G_j)$, so $\nu$ is maxitive. 
\end{proof}

Supposing the continuity of the range $L$ of the maxitive measure enables us to remove the assumption of right-continuity of the family of ideals and gives the converse statement as follows. 

\begin{proposition}\label{nguyen bouchon-meunier2}
Assume that $L$ is a continuous poset. 
A map $\nu : \call{E} \rightarrow L$ is a maxitive measure if and only if there is some family $(\call{I}_t)_{t \in L}$ of ideals of $\call{E}$ such that, for all $G \in \call{E}$, $\{t \in L : G \in \call{I}_t \}$ is a filter with infimum and 
\begin{equation*}
\nu(G) = \bigwedge \left\{t \in L : G \in \call{I}_t \right\}.
\end{equation*}
In this case, $(\call{I}_t)$ is right-continuous  if and only if $\call{I}_t = \{ G \in \call{E} : t \geqslant \nu(G) \}$ for all $t \in L$. 
\end{proposition}

\begin{proof}
If $\nu$ is maxitive, simply take $\call{I}_t = \{ G \in \call{E} : t \geqslant \nu(G) \}$, $t \in L$, which is right-continuous since $L$ is continuous. 
Conversely, assume that Equation~(\ref{ideals0}) is satisfied. Let $\call{J}_t = \bigcap_{s \ssup t} \call{I}_s$. $(\call{J}_t)_{t\in L}$ is a nondecreasing family of ideals of $\call{E}$ such that $\call{J}_t \supset \call{I}_t$ for all $t \in L$. Moreover, $(\call{J}_t)_{t \in L}$ is right-continuous thanks to the interpolation property, and by continuity of $L$ one has $\nu(G) =  \bigwedge\{t \in L : G \in \call{J}_t\}$. Using Proposition~\ref{nguyen bouchon-meunier1}, $\nu$ is maxitive. 

Assume that $(\call{I}_t)$ is right-continuous. The inclusion $\call{I}_t \subset \{ G \in \call{E} : t \geqslant \nu(G) \}$ is clear. If $t \geqslant \nu(G)$, we want to show that $G \in \call{I}_t$, i.e.\ $G \in \call{I}_s$ for all $s \ssup t$. So let $s \ssup t \geqslant \nu(G)$. Equation~(\ref{ideals0}) implies that $G \in \call{I}_s$, and the inclusion $\call{I}_t \supset \{ G \in \call{E} : t \geqslant \nu(G) \}$ is proved. 
\end{proof}


From Proposition~\ref{nguyen bouchon-meunier2} we can deduce the following corollary, which most of the time enables one to extend a maxitive measure to the entire power set $2^{E}$. This is a slight improvement of Heckmann and Huth \cite[Proposition~12]{Heckmann98} and Akian \cite[Proposition~3.1]{Akian99}, the latter being inspired by Maslov \cite[Theorem~VIII-4.1]{Maslov87}. 

Henceforth, $\call{E}^*$ denotes the collection of all $A \subset E$ such that $\{ G \in \call{E} : G \supset A \}$ is a filter. Notice that $\call{E}^*$ is a prepaving containing all singletons, and if $\call{E}$ contains $E$, then $\call{E}^*$ merely coincides with the power set of $E$. 

\begin{proposition}\label{defetoile}
Assume that $L$ is a domain. 
Let $\nu$ be an $L$-valued maxitive measure on $\call{E}$. The map $\nu^* : \call{E}^* \rightarrow L$ defined by 
\begin{equation}\label{extension} 
\nu^*(A) = \bigwedge_{G \in \call{E}, G \supset A} \nu(G)
\end{equation}
is the maximal maxitive measure extending $\nu$ to $\call{E}^*$. 
\end{proposition}


\begin{proof}
If $\nu$ is defined by Equation~(\ref{ideals0}), let $\call{I}^*_t$ denote the collection of all $A \in \call{E}^*$ such that $A \subset B$ for some $B \in \call{I}_t$. 
Then $(\call{I}^*_t)_{t \in L}$ is a nondecreasing family of ideals of $\call{E}^*$ and, for all $A \in \call{E}^*$, $\{t \in L : A \in \call{I}^*_t \} = \bigcup_{G \in \call{E}, G \supset A} \{ t \in L : G \in \call{I}_t \}$ is a filter in $L$. 
Now the fact that $\nu^*(A) = \bigwedge \{ t \in L : A \in \call{I}^*_t \}$ and Proposition~\ref{nguyen bouchon-meunier2} show that $\nu^*$ is maxitive. 
The assertion that $\nu^*$ is the maximal maxitive measure extending $\nu$ to $\call{E}^*$ is not difficult and left to the reader. 
\end{proof}

This corollary also generalises a result due to Kramosil \cite[Theorem~15.2]{Kramosil05a}, where it is assumed that $L$ is a complete chain (which is necessarily a continuous complete semilattice). 
A proof may also be found in \cite{Poncet09} in the general setting of \textit{maxitive maps}.

\section{Cardinal Densities for Maxitive Measures}\label{secmax}

We assume in the remaining part of this paper that $\call{E}$ is a \textit{paving} on $E$, that is a collection of subsets of $E$ containing the empty set, closed under finite unions, covering $E$, and such that, for all $x \in E$, $\{ G \in \call{E} : G \ni x \}$ is nonempty filtered in $\call{E}$ (ordered by inclusion). 

One could certainly think of $\call{E}$ as the base of some topology $\call{G}$ on $E$. 
Also, $\call{E}$ could be thought of as the collection of compact subsets of $E$ when equipped with some topology $\call{O}$ (in which case $\call{G}$ coincides with the power set of $E$), or as the Borel sets of $(E,\call{O})$. 
This variety of examples explains why we do not assume $\call{E}$ be closed under finite intersections. This also highlights why the hypothesis $E \in \call{E}$, adopted by Akian \cite{Akian99}, may be rather restrictive (see the example given above, where $\call{E}$ is the paving of all finite subsets of $E$). 

The collection of (not necessarily Hausdorff) compact subsets of $E$ for the topology $\call{G}$ generated by $\call{E}$ is denoted by $\call{K}$. Note that we always have $\call{E}^* \supset \call{K}$. 

The following theorem gives necessary and sufficient for a maxitive measure to have a density. It goes one step further than  \cite[Theorem~3]{Heckmann98} and \cite[Proposition~3.15]{Akian99}, for we do not need the paving $\call{E}$ to be a topology, and the range $L$ of the maxitive measure to be a (locally) complete lattice. 

\begin{theorem}\label{general}
Assume that $L$ is a domain, and let $\nu$ be an $L$-valued maxitive measure on $\call{E}$. The following conditions are equivalent:
\begin{enumerate}
	\item\label{cmpl1} $\nu$ is completely maxitive, 
	\item\label{gene1} $\nu$ is inner-continuous, i.e.\, for all $G \in \call{E}$, 
	$$
	\nu(G) = \bigoplus_{ K \subset G, K \in \call{K}} \nu^*(K), 
	$$
	\item\label{gene2} $\nu$ has a density. 
\end{enumerate}
If these conditions are satisfied, $\nu$ admits $c^* : x \mapsto \nu^*(\{ x \})$ as maximal density, and $c^*$ is an upper semicontinuous map on $E$. 
\end{theorem}

The concept of upper semicontinuity for poset-valued maps, that we do not recall here, is treated by Penot and Th\'era \cite{Penot82}, Beer \cite{Beer87}, van Gool \cite{vanGool92}, Gerritse \cite{Gerritse97}, Akian and Singer \cite{Akian03}.

\begin{proof}
Fact~1: The restriction of $\nu^*$ to $\call{K}$ admits $c^*$ as cardinal density. Let $K \in \call{K}$ and $m$ be an upper bound of $\{ c^*(x) : x \in K \}$. We want to show that $m \geqslant \nu^*(K)$, so let $s \ssup m$. For any $x \in K$, $s \ssup c^*(x) = \bigwedge_{G \ni x} \nu(G)$, so there is some $G_x \ni x$, $G_x \in \call{E}$, such that $s \geqslant \nu(G_x)$. Since $K$ is compact and $\bigcup_{x \in K} G_x \supset K$, we can extract a finite subcover and write $\bigcup_{j = 1}^k G_{x_j} \supset K$. Thus, $s \geqslant \nu^*(K)$ for any $s \ssup m$, so $m \geqslant \nu^*(K)$ thanks to continuity of $L$. Since $\nu^*(K)$ is itself an upper bound of $\{ c^*(x) : x \in K \}$, this proves that the supremum of $\{ c^*(x) : x \in K \}$ exists and equals $\nu^*(K)$. 

Fact~2: If either $\{\nu^*(K) : K \subset G, K\in\call{K}\}$ or $\{ c^*(x) : x \in G \}$ has a supremum, then 
$$
\bigoplus_{ K \subset G, K \in \call{K}} \nu^*(K) = \bigoplus_{x\in G} c^*(x). 
$$
It suffices to show that both sets have the same upper bounds. Denoting $A^{\uparrow}$ for the set of upper bounds of a subset $A \subset E$, the inclusion $\{\nu^*(K) : K \subset G, K\in\call{K}\}^{\uparrow} \subset \{ c^*(x) : x \in G \}^{\uparrow}$ is due to the fact that $c^*(x) = \nu^*(\{x\})$ and $\{ x \} \in\call{K}$ for any $x \in G$. The equality holds thanks to Fact~1.

Now the implications (\ref{gene1}) $\Rightarrow$ (\ref{gene2}) $\Rightarrow$ (\ref{cmpl1}) are obvious. Let us show that (\ref{gene2}) $\Rightarrow$ (\ref{gene1}). Assume that $\nu(G) = \bigoplus_{x\in G} c(x)$ for all $G\in\call{E}$. Then it is easily seen that $c$ can be replaced by $c^*$ as a density, i.e.\ $\nu(G) = \bigoplus_{x\in G} c^*(x)$ for all $G\in\call{E}$, and the result can be deduced from Fact~2. 

Assume that (\ref{cmpl1}) is satisfied and let $G \in \call{E}$. An upper bound of $\{ c^*(x) : x\in G \}$ is $\nu(G)$. Now let $m$ be an upper bound of $\{ c^*(x) : x\in G \}$. Let $s \ssup m$. The definition of $c^*$ implies that, for all $x \in G$ there is some $G_x \in \call{E}$, $G_x \ni x$, such that $s \ssup \nu(G_x)$. $\call{E}$ is a paving, so there is some $H_x \in \call{E}$ such that $G \cap G_x \supset H_x \ni x$. Since $G = \bigcup_{x \in G} H_x \in  \call{E}$ and $\nu$ is completely maxitive, we deduce that $s \geqslant \nu(G)$. The continuity of $L$ implies that $m \geqslant \nu(G)$, and (\ref{gene2}) is proved. 

To conclude, let us show that $c^*$ is upper semi-continuous, i.e.\ that $\{ t \ssup c^* \}$ is open in $E$ for all $t\in L$. If $t \ssup c^*(x)$, then with the definition of $c^*$ there is some $G \ni x$, which is open in the topology $\call{G}$ generated by $\call{E}$, such that $t \ssup \nu(G)$, which implies that $G \subset \{ t \ssup c^* \}$. 
\end{proof}

Both forthcoming propositions were formulated and proved in \cite{Akian95} in the case where $E$ is a topological space and $L$ is a continuous lattice, see also \cite[Proposition~13]{Heckmann98}. We need to consider $\call{F} = \{ F \subset E : F \in \call{E}^* \mbox{ and } F^c \in \call{E} \}$ and $\call{H} = \call{K} \cap \call{F}$. If one takes the case where $\call{E}$ is a Hausdorff topology, then $\call{F}$ is the collection of closed subsets, and $\call{H} = \call{K}$ is that of compact subsets. 

\begin{proposition}\label{prop:k}
If $L$ is a domain and $\nu$ is an $L$-valued maxitive measure on $\call{E}$, then $\nu$ preserves filtered intersections of elements of $\call{H}$, i.e.\ 
\begin{equation}\label{eqcompactfermes}
\bigwedge_{j\in J} \nu^*(H_j) = \nu^*(\bigcap_{j \in J} H_j), 
\end{equation}
for every filtered family $(H_j)_{j\in J}$ of elements of $\call{H}$ such that $\bigcap_{j \in J} H_j \in \call{H}$. 
\end{proposition}

\begin{proof}
Let $(H_j)_{j \in J}$ be a filtered family of elements of $\call{H}$. If all $H_j$ are nonempty, then this family has nonempty intersection $H$. Indeed, if $H = \emptyset$ and $j_0 \in J$, then $\emptyset = H_{j_0} \cap \bigcap_{j \neq j_0} H_j$, i.e.\ $H_{j_0} \subset \bigcup_{j \neq j_0} H_j^c$. Since $H_j^c \in \call{E}$, we can extract a finite subcover and write $H_{j_0} \subset \bigcup_{i=1}^k H_{j_i}^c$, i.e.\ $\emptyset = H_{j_0} \cap \bigcap_{i=1}^k H_{j_i}$. The family $(H_j)_{j \in J}$ is filtered, so this implies that one of the $H_j$ is empty. 

Now, let us come back to Equality~(\ref{eqcompactfermes}). The set $\{  \nu^*(H_j) : j \in J \}$ admits $\nu^*(H)$ as lower bound. Take another lower bound $m$, and let $G \in \call{E}$ such that $G \supset H$. The family $(H_j \backslash G)_{j \in J}$ is a filtered family of elements of $\call{H}$ with empty intersection, thus $H_j \backslash G = \emptyset$ for some $j \in J$. This implies $\nu^*(H_j)\leqslant \nu(G)$, hence $m \leqslant \nu(G)$ for all $G \supset H$, so that $m \leqslant \nu^*(H)$. We have shown that $\nu^*(H)$ is the greastest lower bound of $\{  \nu^*(H_j) : j \in J \}$. 
\end{proof}

Tightness for maxitive measures can be defined by analogy with tightness for additive measures, so we say that  
an $L$-valued maxitive measure $\nu$ on $\call{E}$ is \textit{tight} if 
$$
\bigwedge_{H \in \call{H}} \nu(H^c) = 0.
$$ 
(In \cite{Puhalskii01}, a tight normed completely maxitive measure on the power set of a topological space is called a \textit{deviability}.)
If $\nu$ is tight, the collection $\call{F}$ can replace $\call{H}$ in Proposition~\ref{prop:k}.  

A \textit{semilattice} is a poset in which every pair $\{s, t\}$ has a least upper bound $s \oplus t$. A \textit{continuous semilattice} is a semilattice which is also a domain. For the following proposition we need to recall that, by \cite[Theorem~III-2.11]{Gierz03}, every continuous semilattice $L$ is \textit{join-continuous} in the sense that, for every $t \in L$ and every filter $F$, 
$
t \oplus \bigwedge F = \bigwedge (t \oplus F). 
$

\begin{proposition}
If $L$ is a continuous semilattice, and $\nu$ is a tight $L$-valued maxitive measure on $\call{E}$, then $\nu$ preserves filtered intersections of elements of $\call{F}$, i.e.\ 
$$
\bigwedge_{j\in J} \nu^*(F_j) = \nu^*(\bigcap_{j \in J} F_j), 
$$
for every filtered family $(F_j)_{j\in J}$ of elements of $\call{F}$ such that $\bigcap_{j \in J} F_j \in \call{F}$. 
\end{proposition}

\begin{proof}
Fix some $H \in\call{H}$, and let $F = \bigcap_{j \in J} F_j$. Then $F_j \cap H$ and $F \cap H$ belong to $\call{H}$, hence $\bigwedge_j \nu^*(F_j \cap H) = \nu^*(F \cap H)$ by Proposition~\ref{prop:k}. Pick some lower bound $m$ of the set $\{ \nu^*(F_j) : j \in J \}$. Thanks to the join-continuity of $L$, $m \leqslant \bigwedge_j (\nu^*(F_j \cap H) \oplus \nu(H^c)) = \nu^*(F \cap H) \oplus \nu(H^c) \leqslant \nu^*(F) \oplus \nu(H^c)$. The tightness of $\nu$ and the join-continuity of $L$ imply $m \leqslant \nu^*(F)$, and the result is proved. 
\end{proof}

\section{Decomposition of Maxitive Measures}\label{decompose}

Here $\call{E}$ is again a paving on $E$. A poset is a \textit{lattice} if every finite subset has a supremum and an infimum. A lattice is \textit{distributive} if finite infima distribute over finite suprema, and \textit{locally complete} if every upper bounded subset has a supremum. A continuous locally complete lattice is a \textit{locally continuous lattice}. A locally continuous lattice which is also distributive is a \textit{locally continuous frame}. Note that every locally continuous frame is a domain. 
Again $\reels_+$, $\overline{\reels}_+$, and $[0, 1]$ are examples of locally continuous frames. 

From Theorem~\ref{general}, the following definition is natural:
\begin{definition}\label{partreg}
Assume that $L$ is a locally continuous lattice, and let $\nu$ be an $L$-valued maxitive measure on $\call{E}$. The \textit{regular part} of $\nu$ is the map defined on $\call{E}$ by
\begin{eqnarray*}
\lfloor \nu \rfloor(G) = \bigoplus_{K \in\call{K}, K\subset G}  \nu^*(K).
\end{eqnarray*} 
\end{definition} 

The regular part of $\nu$ is a completely (or regular) maxitive measure on $\call{E}$ with density $c^*$. This is the greatest completely maxitive measure lower than $\nu$ on $\call{E}$. Moreover, $\lfloor \nu \rfloor^*$ and $\nu^*$ coincide on $\call{K}$, hence $\lfloor\lfloor \nu \rfloor\rfloor = \lfloor \nu \rfloor$. 

The following theorem states the existence of a \textit{residual part} $\bot\nu$ of a maxitive measure $\nu$. 

\begin{theorem}\label{prop:singpart}
Assume that $L$ is a locally continuous frame, and let $\nu$ be an $L$-valued maxitive measure on $\call{E}$. There exists a smallest maxitive measure $\bot\nu$ on $\call{E}$, called the \textit{residual part} of $\nu$, such that the decomposition 
$$
\nu = \lfloor \nu \rfloor \oplus \bot\nu
$$
holds. Moreover, $\bot\nu$ coincides with its own residual part, i.e.\ $\bot(\bot\nu) = \bot\nu$, and the residual part of the regular part of $\nu$ equals $0$, i.e.\ $\bot \lfloor \nu \rfloor = 0$. 
\end{theorem}

\begin{proof}
We give a constructive proof for the existence of $\bot\nu$. 
Let $\bot\nu(G) = \bigwedge\{ t \in L : G \in \call{I}_t \}$, where $\call{I}_t := \{ G \in \call{E} : \forall H \subset G, \nu(H) \leqslant \lfloor \nu \rfloor(H) \oplus t\}$. Then $(\call{I}_t)_{t \in L}$ is a nondecreasing family of ideals of $\call{E}$, and distributivity implies that $\{ t \in L : G \in \call{I}_t \}$ is a filter, 
for every $G \in \call{E}$. From Proposition~\ref{nguyen bouchon-meunier2}, we deduce that $\bot\nu$ is a maxitive measure. 

Since $\nu(G) \in \call{I}_t$ for $t = \nu(G)$, we have $\nu(G) \geqslant \bot\nu(G)$, thus $\nu \geqslant \lfloor \nu \rfloor \oplus \bot\nu$. Let us prove that the reserve inequality holds. Let $G \in \call{E}$, let $m$ be an upper bound of the pair $\{ \lfloor \nu \rfloor (G), \bot\nu (G) \}$, and let $u \ssup m$. There is some $t \in L$, $\nu(G) \leqslant \lfloor \nu \rfloor(G) \oplus t$, such that $u \geqslant t$. Hence, $u \geqslant \nu(G)$, so by continuity of $L$, $m \geqslant \nu(G)$, and the reserve inequality is proved. 

To show that $\bot(\bot\nu) = \bot\nu$, first notice that $\bot(\bot\nu) \leqslant \bot\nu$, since $\bot\nu = \lfloor \bot\nu \rfloor \oplus \bot(\bot\nu)$. Second, $\lfloor \bot\nu \rfloor$ has a density and is lower than $\nu$, hence is lower than $\lfloor \nu \rfloor$. Thus, $\nu = \lfloor \nu \rfloor \oplus \bot\nu = \lfloor \nu \rfloor \oplus \lfloor \bot\nu \rfloor \oplus \bot(\bot\nu) = \lfloor \nu \rfloor \oplus \bot(\bot\nu)$. This implies that $\bot\nu \leqslant \bot(\bot\nu)$.  The fact that $\bot\lfloor \nu \rfloor = 0$ is straightforward. 
\end{proof}

See also \cite{Poncet09} for a proof relying on purely order-theoretical properties of the set of maxitive measures. As a consequence of the previous result we have the following corollary. 

\begin{corollary}
Assume that $L$ is a locally continuous frame, and let $\nu$ be an $L$-valued maxitive measure on $\call{E}$. Then $\nu$ has a density if and only if $\nu = \lfloor \nu \rfloor$ if and only if $\bot\nu = 0$. 
\end{corollary}

It is worth summarizing calculus rules that apply to operators $\lfloor \cdot \rfloor$, $\bot \cdot$, and $(\cdot)^*$: 

\begin{proposition}
Assume that $L$ is a locally continuous frame, and let $\nu, \tau \in \call{M}$. Then the following properties hold:
\begin{enumerate}
	\item $\nu = \lfloor \nu \rfloor \oplus \bot\nu$, 
	\item $\nu = \lfloor \nu \rfloor \Leftrightarrow \bot\nu = 0$, 
	\item $\lfloor \lfloor \nu \rfloor \rfloor = \lfloor \nu \rfloor$, 
	\item $\lfloor \nu \oplus \tau \rfloor = \lfloor \nu \rfloor \oplus \lfloor \tau \rfloor$, 
	\item $(\nu \oplus \tau)^* = \nu^* \oplus \tau^*$, 
	\item $\bot(\bot\nu) = \bot\nu$, 
	\item $\bot(\nu \oplus \tau) \leqslant \bot \nu \oplus \bot \tau$, 
	\item $\bot \lfloor \nu \rfloor = 0$. 
\end{enumerate}
Moreover, if $\call{E}$ is a topology, we have 
\begin{enumerate}
	\item $\lfloor \nu^* \rfloor = \lfloor \nu \rfloor^*$, 
	\item $\bot(\nu^*) \leqslant (\bot \nu)^*$. 
\end{enumerate}
\end{proposition}

Among the previous list one could worry about some desirable property missing. One naturally expects that the regular part of a residual part be equal to zero ($\lfloor \bot\nu \rfloor = 0$), or, in other words, 
that the residual part to be null on compact subsets, or at least on $\call{H}$. However, for the latter to be realized we need some additional conditions on $\call{E}$, namely that $\call{E}$ be closed under the formation of complements ($\call{E}$ is then called a \textit{Boolean algebra}). Hence we say that a maxitive measure $\nu$ is \textit{singular} if $\nu^*(H) = 0$ for all $H \in \call{H}$. 

\begin{theorem}\label{boolean}
Assume that $\call{E}$ is a Boolean algebra on $E$ and $L$ is a locally continuous frame. Let $\nu$ be an $L$-valued maxitive measure on $\call{E}$. Then the map defined on $\call{E}$ by 
\begin{equation}\label{singpart}
\nu_s(G) = \bigwedge_{H \in \call{H}, H \subset G} \nu(G \setminus H) 
\end{equation}
is the greatest $L$-valued singular maxitive measure such that the decomposition $\nu = \lfloor \nu \rfloor \oplus \nu_s$ holds. 
\end{theorem}

\begin{remark}
By Theorem~\ref{prop:singpart}, $\bot\nu$ is lower than $\nu_s$, hence singular when $\call{E}$ is a Boolean algebra. 
\end{remark}

\begin{proof}
We prove that $\nu_s$ defined by Equation~(\ref{singpart}) is maxitive and that the decomposition $\nu = \lfloor \nu \rfloor \oplus \nu_s$ holds. Let $G\in\call{E}$. If $H \subset G$, $H \in\call{H}$, we have, in view of the distributivity of finite joins with respect to arbitrary meets, $\nu(G) = \nu(H) \oplus \nu(G \setminus H) \leqslant \lfloor \nu \rfloor(G) \oplus \nu(G \setminus H) \leqslant \lfloor \nu \rfloor(G) \oplus \nu_s(G)$. The converse inequality is obvious, hence $\nu = \lfloor \nu \rfloor \oplus \nu_s$. 

Next we show that $\nu_s$ is maxitive. If $G, G' \in \call{E}$, 
\begin{eqnarray*}
\nu_s(G) \oplus \nu_s(G') &=& \bigwedge_{H \subset G, H' \subset G'} \nu(G \setminus H) \oplus \nu(G' \setminus H') \\
&\geqslant& \bigwedge_{H \subset G, H' \subset G'} \nu((G \cup G') \setminus (H \cup H')) \\
&\geqslant& \nu_s(G \cup G') \mbox{ since } H \cup H' \in \call{H}, 
\end{eqnarray*}
and it remains to show that $\nu_s$ is nondecreasing. 
So let $G, G' \in \call{E'}$ such that $G \subset G'$, and let $H' \subset G'$, $H' \in \call{H}$. Then $\nu(G' \setminus H') \geqslant \nu(G \setminus H') = \nu(G \setminus (G \cap H')) \geqslant \nu_s(G)$, the last inequality coming from the fact that $G \cap H' \in \call{H}$ since $\call{E}$ is stable under complementation. We deduce that $\nu_s(G') \geqslant \nu_s(G)$, and $\nu_s$ is maxitive. 

Also, $\nu_s$ is singular since, for all $H \in \call{H}$, $\nu_s^*(H) = \nu_s(H) = 0$. 

Suppose that $\nu = \lfloor \nu \rfloor \oplus \tau$ for some singular maxitive measure $\tau$, and let $G \in \call{E}$. Then, for all $H \subset G$, $H \in \call{H}$, $\tau(G) = \tau(G \setminus H) \oplus \tau(H) = \tau(G \setminus H)$ by singularity of $\tau$. Hence $\tau(G) \leqslant \nu(G \setminus H)$, and we get $\tau(G) \leqslant \nu_s(G)$ for all $G \in \call{E}$. 
\end{proof}

\begin{corollary}
If $\call{E}$ is a Boolean algebra on $E$ and $L$ is a locally continuous frame, then every $L$-valued tight maxitive measure on $\call{E}$ has a density. 
\end{corollary}

\begin{proof}
Let $\nu$ be a tight maxitive measure on $\call{E}$. Since $\bot\nu \leqslant \nu_s$, we have in particular, for all $G \in \call{G}$, $\bot\nu(G) \leqslant \nu_s(E) = \bigwedge_{H \in \call{H}} \nu(H^c) = 0$. This means that the singular part of $\nu$ is null, so that $\nu = \lfloor \nu \rfloor$, and $\nu$ has a cardinal density thanks to Proposition~\ref{general}. 
\end{proof}

\begin{acknowledgements}
I would like to thank Marianne Akian for her valuable comments and suggestions. 
\end{acknowledgements}

\bibliographystyle{plain}

\def\cprime{$'$} \def\cprime{$'$} \def\cprime{$'$} \def\cprime{$'$}
  \def\ocirc#1{\ifmmode\setbox0=\hbox{$#1$}\dimen0=\ht0 \advance\dimen0
  by1pt\rlap{\hbox to\wd0{\hss\raise\dimen0
  \hbox{\hskip.2em$\scriptscriptstyle\circ$}\hss}}#1\else {\accent"17 #1}\fi}
  \def\ocirc#1{\ifmmode\setbox0=\hbox{$#1$}\dimen0=\ht0 \advance\dimen0
  by1pt\rlap{\hbox to\wd0{\hss\raise\dimen0
  \hbox{\hskip.2em$\scriptscriptstyle\circ$}\hss}}#1\else {\accent"17 #1}\fi}

\end{document}